\documentclass[12pt]{article}
\usepackage{amsmath}
\usepackage{amscd}
\usepackage{amsthm}
\usepackage{amssymb} \usepackage{latexsym}
\usepackage{eufrak}
\usepackage{euscript}
\usepackage[width=18cm,height=24cm]{geometry}
\usepackage{epsfig}
\usepackage{tikz}
\usepackage{graphics}
\usepackage{array}
\usepackage{enumerate}
\usepackage{authblk}
\theoremstyle{theorem}
\newtheorem{theorem}{Theorem}[section]
\theoremstyle{corollary}
\newtheorem{corollary}{Corollary}[section]
\theoremstyle{lemma}

\theoremstyle{definition}

\theoremstyle{proof}

\theoremstyle{remark}

\newcommand{\bel}[1]{\begin{equation}\label{#1}}

\newcommand{\be}{\begin{equation}}

\newcommand{\ba}{\begin{eqnarray}}
\newcommand{\ea}{\end{eqnarray}}
\newcommand{\rf}[1]{(\ref{#1})}
\newcommand{\bi}{\bibitem}
\newcommand{\qe}{\end{equation}}

\begin{document}
\setcounter{Maxaffil}{2}

\title{Effect on normalized graph Laplacian spectrum by motif attachment and duplication}

\author[1]{\rm Ranjit Mehatari}
\author[1,2]{\rm Anirban Banerjee}
\affil[1]{Department of Mathematics and Statistics}
\affil[2]{Department of Biological Sciences}
\affil[ ]{Indian Institute of Science Education and Research Kolkata}
\affil[ ]{Mohanpur-741246, India}
\affil[ ]{\textit {\{ ranjit1224,anirban.banerjee\}@iiserkol.ac.in}}
\date{}
\maketitle
\begin{abstract}

To some extent, graph evolutionary mechanisms can be explained by its spectra. Here, we are interested in two graph operations, namely, motif (subgraph) doubling and attachment that are biologically relevant. We investigate how these two processes affect the spectrum of the normalized graph Laplacian. A high (algebraic) multiplicity of the eigenvalues $1, 1\pm 0.5, 1\pm \sqrt{0.5}$ and others has been observed in the spectrum of many real networks. We attempt to explain the production of distinct eigenvalues by motif doubling and attachment. Results on the eigenvalue $1$ are discussed separately.
\end{abstract}


\section{Introduction} 
Nowadays, spectral graph theory is playing an important role to analyze the structure of real networks \cite{BanerjeeJost2009, BanerjeeJost2007,  GkantsidisMihailZegura2003, Mieghem2011, NadakuditiNewman2012,WS}. 
The underlying graph of biological and other real networks evolves with time. 
The evolutionary mechanisms may lead to the construction of certain local structures (motif)  that could be described by different eigenvalues of normalized graph Laplacian \cite{BanerjeeJosta, Vukadinovic2002}. 
Duplication of a group of genes \cite{HuynenBork1998,Wagner2003}  and horizontal \cite{Dutta02, Jain02} gene transfer may cause the existence of repetitive-motif and attachment of a distinct small network, respectively, in an existing biological network \cite{UriAlon, ShellmanEtAl, Sole, Wagner2003}.
Two graph operations, motif doubling and attachment of a smaller graph into the existing graph, are interest of our study. 
Here we intensively investigate the emergence of particular eigenvalues such as eigenvalue 1 and other  by the above-mentioned graph operations.\\

Let $\Gamma=(V,E)$ be a simple, connected, finite graph of order $n$ with the vertex set $ V$ and the edge set $E$. Two vertices $i,j
\in V(\Gamma)$, are connected by an edge in $E(\Gamma)$, are called neighbors, $i\sim j$. Let $n_i$ be the degree of $i \in V(\Gamma)$, that is, the number of neighbors of $i$.
For the function $g:V(\Gamma)\rightarrow \mathbb{R}$ we define the normalized graph Laplacian as
\bel{int:equ1}
\Delta g(x):=  g(x) -\frac{1}{n_x} \sum_{y, y \sim x}g(y) .
\qe
Note that, this operator is different from the (algebraic) graph Laplacian operator, 
$Lg(x):=n_x g(x)-\sum_{y, y \sim x}g(y)$ (see \cite{Guo,Merris,Pasten} for this operator), but  is similar to the Laplacian,
$\mathcal{L} g(x):=  g(x) -\sum_{y, y \sim x} \frac{1}{\sqrt{n_{x} n_{y}}}g(y)$
 investigated in \cite{Chung} and thus, both have the same spectrum.
 Now we recall some of the basic properties of  eigenvalues and eigenfunctions of the operator (\ref{int:equ1}) from \cite{BanerjeeJost2008a}. The normalized Laplacian is
symmetric for the product,
\bel{int:equ8}
(g_1,g_2):=\sum_{i \in V} n_i g_1(i)g_2(i),
\end{equation}
for real valued functions $g_1,g_2$ on $V(\Gamma)$. 
Since $(\Delta g,g)\ge 0$, all eigenvalues of $\Delta$ are non-negative.
  The eigenvalue equation of $\Delta$ is
  \bel{int:equ2}
\Delta f - \lambda f = 0,
\end{equation}
 where a non zero solution $f$ is called an eigenfunction corresponding to the eigenvalue $\lambda$. 
If we arrange all the eigenvalues in a non-decreasing manner we have,
$\lambda_0=0 < \lambda_1 \le ... \le \lambda_{n-1}\le 2,$
 with 
 $\lambda_{n-1}= 2$ \textit{iff} the graph is bipartite. For a connected graph, the smallest eigenvalue is $\lambda_0=0$ with a constant eigenfunction. Since all the eigenfunctions are orthogonal to each other, for any eigenfunction $f$, we have
 \bel{int:equ9}
\sum_{i \in V} n_{i} f(i)=0.
\end{equation}
 For a graph with $N$ vertices, $\lambda_1\leq\frac{N}{N-1}\leq \lambda_{N-1} $ and the equality hold \textit{iff} the graph is complete, that is, for a complete graph
$\lambda_1=\lambda_2=\cdots=\lambda_{N-1}=\frac{N}{N-1}.$
 Let $m_\lambda$ be the algebraic multiplicity of the eigenvalues $\lambda$.
 The eigenvalue equation \rf{int:equ2} becomes
\bel{int:equ3}
\frac{1}{n_i}\sum_{j \sim i} f(j)=(1-\lambda)f(i) \text{  }\forall i\in\Gamma.
\qe
In particular, if an eigenfunction $f$ vanishes at $i$, then  $\sum_{j\sim
i}f(j) = 0$, and conversely (for eigenvalue 1, the converse is not always true).
For  $\lambda=1$, equation (\ref{int:equ3}) becomes
\bel{int:equ4}
\sum_{j\in \Gamma, j \sim i} f(j)=0 \text{   } \forall i\in \Gamma,
\qe
which is a special property of an eigenfunction for the eigenvalue 1.
%
Note that, the nullity of the adjacency matrix  \cite{dr} $A$ of $\Gamma$ is the same as the algebraic multiplicity of eigenvalue $1$ in  $ \Delta$. That is,
\bel{int:equ7}
m_1=nullity \text{ }of \text{ }A.
\qe
The nullity of the adjacency matrix was very much studied in earlier mathematical Works (see the survey by Gutman and Borovi\'canin \cite {Gutman})
Here, we take a different approach to study the eigenvalue $1$ in the context of normalized graph Laplacian.\\


 
 Now, we extend the discussion on the production  of the eigenvalue 1 
 investigated in  \cite{BanerjeeJost2008a} and  generalize the results to a broad range of  operations. 
\subsubsection*{Vertex Doubling:}
Doubling of a vertex $p$ of $\Gamma$ is to add a vertex $q$ to $\Gamma$ and connect it to all $j$ in $\Gamma$, whenever $j\sim p$. Vertex doubling, of a vertex $p$ of $\Gamma$, ensures the eigenvalue 1 with an eigenfunction  $f_1$ that takes value 1 at $p$, -1 at its double and 0 otherwise \cite{BanerjeeJost2008a}.
Now, if we double the vertex $p$, $m$ times, then the resultant graph possesses the eigenvalue 1 with the  multiplicity at least $m$ with the corresponding eigenfunctions,
\bel{eig1:equ2}
 f^{(i)}_j(x)=\begin{cases} 1 
 & \text{  if } x=p,q_1,q_2,...,q_{j-1}\\
 -j&\text{ if } x=q_j\\
 0 &\text{ else},
\end{cases}
\qe 
for $j$=1,2,...,$i$ and  $i$=1,2,...,$m$; where $q_1$, $q_2$, $q_3$,$\cdots$,$q_m$ are the vertices produced by repeated-doubling of  $p$.
\subsubsection*{Motif Doubling:}
Let $\Sigma$ be a connected induced subgraph of $\Gamma$ with vertices $p_1,\dots ,p_m$. 
Let $\Gamma^\Sigma$ be obtained from $\Gamma$ by adding a copy of the
motif $\Sigma$ consisting of the vertices
$q_1,\dots ,q_m$ and the corresponding connections between them, and connecting each $q_\alpha$ with all $p \notin
\Sigma$ that are neighbors of $p_\alpha$. Now, if $\Sigma$ has an eigenvalue 1 with an eigenfunction $f^{\Sigma}_1$, then  
$\Gamma^\Sigma$ also ensures an eigenvalue $1$ with the eigenfunction 
\bel{eig1:equ3}
f^{\Gamma^\Sigma}_1(p)=\begin{cases} f^\Sigma_1(p_\alpha)
&  \text{ if } p=p_\alpha \in \Sigma \\
-f^\Sigma_1(p_\alpha)
  &\text{ if } p=q_\alpha\\
0 &\text{ else},
\end{cases}
\qe
where $q_\alpha$
is the double of $p_\alpha\in\Sigma$ \cite{BanerjeeJost2008a}.
Now the above operation can easily  be extendable for doubling $\Sigma$ repeatedly $m$ times. Let  
$\Sigma^1, \Sigma^2, \dots, \Sigma^m$  be the doubles of $\Sigma$ and the resultant graph is $\Gamma^{\Sigma^m}$, where  $q_\alpha^{(m)}\in \Sigma^m$ is the double of $p_\alpha$.
\begin{theorem}
\label{eig1:th1}
If $\Sigma$ has an eigenvalue 1 with an eigenfunction $f^{\Sigma}_1$, then  
$\Gamma^{\Sigma^m}$ also ensures an eigenvalue 1 with multiplicity at least  $m$.
\end{theorem}
\begin{proof}
For each $j={1,2,...,m}$ we have the eigenfunctions 
\bel{eig1:equ4}
f^{\Gamma^{\Sigma^m}}_j(p)=\begin{cases} f^\Sigma_1(p_\alpha)
 & \text{ if } p=p_\alpha \in \Sigma \\
f^\Sigma_1(p_\alpha)
  &\text{ if } p=q^{(l)}_\alpha,j>1, 1\le l\le {j-1}\\
  -jf^\Sigma_1(p_\alpha)
&  \text{ if } p=q^{(j)}_\alpha\\
0 &\text{ elsewhere,}
\end{cases}
\qe
 corresponding to eigenvalue 1.
\end{proof}
\subsubsection*{Graph Attachment:}\label{eig1:GraphCoupling}
Let $\Gamma_1$ be a graph with an eigenfunction $f^{\Gamma^1}_1$ corresponding to an eigenvalue 1. We take a vertex $q\in\Gamma$ and another vertex $p\in\Gamma_1$ and construct $\Gamma_0=\Gamma\cup\Gamma_1$ in such a way that $\Gamma_0$ contains all the connections of $\Gamma$ and $\Gamma_1$ and in addition, we add edges from $q$ to $j\in \Gamma_1$, whenever $j\sim p$.
\begin{theorem}
\label{eig1:th2}
$\Gamma_0$ possesses the eigenvalue 1 with an eigenfunction $f^{\Gamma_0}_1$.
\end{theorem}
\begin{proof}
We choose
\bel{eig1:equ5}
f^{\Gamma_0}_1(i)=\begin{cases}f^{\Gamma^1}_1(i)&\text{ if } i\in\Gamma_1\\0&\text{ if }i\in\Gamma.
\end{cases}
\qe
\end{proof}
\paragraph*{Example:}
Let $\Gamma$ be a triangle and $\Gamma_1$ be a chain of three vertices. Unlike $\Gamma$, $\Gamma_1$ has an  eigenvalue 1.
Now, we join any vertex from $\Gamma$ to the both end vertices of $\Gamma_1$ and produce $\Gamma_0$ that possesses eigenvalue 1.
%
Note that, if $m_1=l$ in $\Gamma_1$, $\Gamma_0$ ensures the eigenvalue $1$ with the multiplicity at least $l$.

\section{Eigenvalue $\lambda$} 
In this section, we describe how the graph operations,  motif doubling and attachment of another graph produce a particular eigenvalue $\lambda$. 

\subsubsection*{Motif Doubling :}
Let $\Sigma$ be a motif in $\Gamma$. Suppose a real valued function $f$ satisfies
\bel{eig_lambda:equ1}
\frac{1}{n_i}\sum_{j\in \Sigma, j \sim i} f(j)=(1-\lambda)f(i) \text{
  for all }i\in \Sigma \text{ and some real value } \lambda,
\qe
where $n_i$ is the degree of $i$ in $\Gamma$.
Let $\Gamma^\Sigma$ be the graph obtained from $\Gamma$ by doubling $\Sigma$ in $\Gamma$ \cite{BanerjeeJost2008a}. Then $\Gamma^\Sigma$ ensures an eigenvalue $\lambda$ with the eigenfunction 
\bel{eig_lambda:equ2}
f^{\Gamma^{\Sigma}}_\lambda(p)=\begin{cases} f(p_\alpha)
  &\text{ if } p=p_\alpha \in \Sigma \\
  -f(p_\alpha)
  &\text{ if } p=q_\alpha, \\
0 &\text{ else}.
\end{cases}
\qe

\begin{theorem}
\label{eig_lambda:th1}
$\Sigma$, is a motif in a graph $\Gamma$, 
 consists of the vertices $p_1,p_2,...,p_m$. If $\Gamma^\Sigma$ is obtained by doubling $\Sigma$ in $\Gamma$, then the additional eigenvalues of the graph $\Gamma^\Sigma$ are given by 
$A_\Sigma+(\lambda-1)D_\Sigma = 0$, where $A_\Sigma$ is the adjacency matrix of $\Sigma$ and $D_\Sigma$ is the diagonal matrix with the degrees of $p_1,p_2,...,p_m$.

\end{theorem}
\begin{proof}
From the eigenvalue equation (\ref{eig_lambda:equ1}) we get,
\bel{eig_lambda:equ5}
(\lambda-1)f(p_i) + \sum_{p_j\in \Sigma, p_j \sim p_i} f(p_j)= 0,  \text{ }\forall p_i\in \Sigma,
\qe
which gives a set of  $m$ homogeneous liner equations that can be written as:
\bel{eig_lambda:equ9}
AF=O, \qe
where $A= A_\Sigma+(\lambda-1)D_\Sigma$,  
$F=(f(p_1)\text{ } f(p_2) \text{ }\cdots \text{ }f(p_m))^T$ and $O=(0\text{ }0\text{ }\cdots \text{ }0)^T.$
Now, to get a nonzero solution for $f$, we  have
\bel{eig_lambda:equ10}
det(A)=0.
\qe
Hence the proof.

%
%
\end{proof}
\paragraph*{Example:} Let $\Gamma$ be a graph with more than three vertices and $\Sigma$ be a chain, with the vertices, $p_1-p_2-p_3$. Now,  from the theorem \ref{eig_lambda:th1} we have\\
\[ \left|\begin{array}{ccc}
    n_{p_1}(\lambda-1) & 1 & 0 \\
    1 & n_{p_2}(\lambda-1) & 1 \\
    0 & 1 & n_{p_3}(\lambda-1)
\end{array} \right|=0,\]\\
which gives
$\lambda=1,1\pm\sqrt{\frac{n_{p_1}+n_{p_3}}{n_{p_1}n_{p_2}n_{p_3}}}$.
Note that, if $n_{p_1}=n_{p_2}=n_{p_3}=k$, then  $\lambda=1, 1\pm \frac{\sqrt{2}}{k}$, that is, 
$\lambda\approx1$ when $k$ is very large.

\subsubsection*{Repeated Duplication of a Motif}
The repeated doubling of a motif increases the multiplicity of the eigenvalue  $\lambda$.

\begin{theorem}
\label{eig_lambda:th2}
Let  $\Sigma$, consists of the vertices $p_1$, $p_2$, $\ldots$, $p_m$, be a motif in $\Gamma$. Suppose a nonzero real valued function $f$ satisfies the equation 
\bel{eig_lambda:equ13}
\frac{1}{n_{p_i}}\sum_{p_j\in \Sigma, p_j \sim p_i} f(p_j)=(1-\lambda)f(p_i) \text{,
  for all }p_i\in \Sigma \text{ and some real value } \lambda,
\qe
where $n_{p_i}$ is the degree of the vertex $p_i$ in $\Gamma$.
If $\Gamma^{\Sigma^m}$ is the graph obtained from $\Gamma$ by repeatedly doubling $\Sigma$ $m$ times, as we described in theorem \ref{eig1:th1},
then $\Gamma^{\Sigma^m}$ ensures the eigenvalue $\lambda$ with the multiplicity at least $m$.
\end{theorem}
\begin{proof}
We construct $m$ eigenfunctions:
\bel{eig_lambda:equ14}
f^{\Gamma^{\Sigma^m}}_j(p)=\begin{cases} f(p_\alpha)
  &\text{ if } p=p_\alpha \in \Sigma \\
f(p_\alpha)
  &\text{ if } p=q^{(l)}_\alpha, j>1, 1\le l\le {j-1}\\
  -jf(p_\alpha)
  &\text{ if } p=q^{(j)}_\alpha\\
0 &\text{ else}
\end{cases}
\qe
$j=1,2,..., m$ corresponding to the eigenvalue $\lambda$, in $\Gamma^{\Sigma^m}$.
\end{proof}

\paragraph{Example (star of triangles):} Let $\Gamma^i$ denotes the star with $i$ triangles, that is, $i$ triangles are joined at a single vertex. 
Theorem \ref{eig_lambda:th1} and theorem \ref{eig_lambda:th2} show that the eigenvalues of $\Gamma^i$ are $0$, $0.5$ and $1.5$ with multiplicities $1$, $i-1$ and $i+1$ respectively, whereas the theorem \ref{eig_lambda:th2} provides the corresponding eigenfunctions.

\subsubsection*{Graph Attachment :}
Now, we discuss about the production of an particular eigenvalue $\lambda$  due to graph attachment, that is, we add a graph $\Sigma$ to an existing graph $\Gamma$ (the spectrum of algebraic graph Laplacian produced by a particular graph attachment operation has been discussed in \cite{Pasten}).
 Here, the attachment operation is different than what we described in section (\ref{eig1:GraphCoupling}).
 Let $\Sigma_c$ be an induced subgraph of $\Sigma$. We take a vertex $p\in \Gamma$ and join $p$ to all $j\in \Sigma_c$ by an edge and get the resultant graph $\Gamma^\Sigma$. Now, we have the following theorem associated to this operation.
\begin{theorem}
\label{eig_lambda:th3}
If there exists a nonzero real valued function f which satisfies the equation

\bel{eig_lambda:equ15}
\frac{1}{n_i}\sum_{j\in \Sigma, j \sim i} f(j)=(1-\lambda)f(i) \text{,
  for all }i\in \Sigma \text{ and some real value } \lambda
\qe
and
\bel{eig_lambda:equ16}
\sum_{j\in \Sigma_c}f(j)=0,
\qe
where $n_i$ is the degree of $i$ in $\Gamma^\Sigma$,
then $\Gamma^\Sigma$ possesses the eigenvalue $\lambda$ with an eigenfunction $f^{\Gamma^\Sigma}$ which coincides with f on $\Sigma$.
\end{theorem}
\begin{proof}
We take,
\begin{center}
$f^{\Gamma^\Sigma}(p)=\begin{cases}
f(p) &\text{ if } p\in \Sigma\\
0 &\text{ otherwise.}
\end{cases}$
\end{center}
\end{proof}

\begin{corollary}
\label{eig_lambda:cor1}
Let  $\Sigma$ be a regular graph. We  join a vertex $p\in \Gamma$ to all the vertices in $\Sigma$ and get a new graph $\Gamma^\Sigma$. Now, if there exists a nonzero real valued function f which satisfies the equation $\frac{1}{n_i}\sum_{j\in \Sigma, j \sim i} f(j)=(1-\lambda)f(i) \text{,
  for all }i\in \Sigma \text{ and some real value } \lambda$ (where $n_i$ is the degree of $i$ in $\Gamma^\Sigma$), then $\Gamma^\Sigma$ possesses the eigenvalue $\lambda$ with the eigenfunction $f^{\Gamma^\Sigma}$ which coincides with f on $\Sigma$. 
\end{corollary}
\begin{proof}
Here $\Sigma_c=\Sigma$ and $\langle f^{\Gamma^\Sigma},f_0 \rangle=0$, where $f_0$ is a constant eigenfunction.
\end{proof}

\begin{corollary} 
\label{eig_lambda:cor2}
The above corollary \ref{eig_lambda:cor1} also holds if $\Sigma$ is the complete graph $K_n$. Moreover, 
  here $\lambda=\frac{n+1}{n}$ with $m_\lambda=n-1$.
\end{corollary}
\begin{proof}
Since a complete graph is regular, the first part of the corollary follows from the corollary \ref{eig_lambda:cor1}.
Now, if we consider the graph $\Sigma^\prime=\Sigma\cup\lbrace p\rbrace$ obtained by taking all the edges of $\Sigma$ and the edges $(i,p)$ for all $i\in \Sigma$, then $\Sigma^\prime$ becomes $K_{n+1}$ that has the eigenvalue $\lambda=\frac{n+1}{n}$ with $n-1$ eigenfunctions $f^i_\lambda$, $i=1,2,\ldots,n-1$ which takes the value 0 at $p$, and $\sum_{j\in \Sigma}f^i_\lambda(j)=0$ for $i=1,2,\ldots,n-1$. 
Thus, each $f^i_\lambda$ can be extended to an eigenfunction of $\Gamma^\Sigma$. Hence the proof.
\end{proof}
\paragraph{Example:} If we consider $\Gamma$ and $\Sigma$ both to be $K_2$, then by the above corollary we get $\lambda=1.5$ with an eigenfunction $f_{1.5}=\begin{bmatrix}
0&0&1&-1
\end{bmatrix}^T.$

\paragraph*{Example (Kite $K(m,n)$): }
A \textit{kite} $K(m,n)$ is a graph with $m+n+1$ vertices obtained by joining all the vertices of $K_m$ and $K_n$ to a single vertex $\alpha$. Thus $K(m,n)$ has $m$ vertices, $p_1, \dots, p_m$, with degree $m$, $n$ vertices $q_1, \dots, q_m$,  with degree $n$, and  a vertex $\alpha$ of degree $m+n$. 
Thus, $K(m,n)-\alpha =K_m\cup K_n.$
Now, the corollary \ref{eig_lambda:cor2} shows that $K(m,n)$ has eigenvalues $ \frac{m+1}{m}$ and $\frac{n+1}{n}$ with the  multiplicity $ m-1 $ and $ n-1 $ respectively.
The rest two eigenvalues are $\lambda_\beta$ and $\lambda_\gamma$, such that, 
$\lambda_\beta +\lambda_\gamma =1+\frac{1}{m}+\frac{1}{n}.$
%
%
%
%
Note that, using theorem \ref{eig_lambda:th1} we get $ K(m,m) $ has the eigenvalues $0$, $ \frac{m+1}{m}$ and  $\frac{1}{m}$ with the multiplicity $1$, $2m-1$, and  $1$ respectively.

\begin{corollary} 
\label{eig_lambda:cor3}
If we join all vertices of $\Sigma$ to the vertices $p_1,p_2,\ldots p_k\in \Gamma$, 
then the theorem \ref{eig_lambda:th3} also holds trivially.
\end{corollary}
\begin{proof}
We take the eigenfunction
\begin{center}
$f^{\Gamma^\Sigma}(p)=\begin{cases}
f(p) &\text{ if } p\in \Sigma\\
0 &\text{ otherwise}
\end{cases}$
\end{center}
corresponding to the eigenvalue $\lambda$.
\end{proof}

\begin{corollary}
\label{eig_lambda:cor4}
Now we  combine corollary \ref{eig_lambda:cor2}  and corollary \ref{eig_lambda:cor3}, that is, we take $\Sigma$ to be $K_n$ and we obtain $\Gamma^\Sigma$ by joining all vertices $\Sigma$ to the vertices $p_1,p_2,\ldots,p_k\in \Gamma$. Now, the resultant graph posses the eigenvalue  $\lambda=\frac{n+k}{n+k-1}$  with  $m_\lambda=n-1$. 
\end{corollary}
\begin{proof}
The proof is similar  in  corollary \ref{eig_lambda:cor2}.
Here $\Sigma^\prime=\Sigma\cup\lbrace p_1,p_2,\ldots ,p_k\rbrace$ obtained by taking all the edges of $\Sigma$ and the edges $(i, p_j)$, for all $i\in \Sigma$ and  $j=1,2,\ldots,k$. 
Thus the $\Sigma^\prime$  becomes $K_{n+k}$ which has the eigenvalue $\lambda=\frac{n+k}{n+k-1}$ with $n-1$ eigenfunctions $f^i_\lambda$, $i=1,2,\ldots,n-1$ that take the value 0 at $p_1, p_2, \ldots, p_k$, and $\sum_{j\in \Sigma}f^i_\lambda(j)=0$ for $i=1, 2, \ldots, n-1$.
Since $f^i_\lambda (p_j)=0$ for $j=1,2,\ldots,k$, every $f^i_\lambda$ can be extended to an eigenfunction of $\Gamma^\Sigma$ for $i=1,2\ldots ,n-1$. 
\end{proof}
Note that, if we delete any edge $(p_l, p_j)$ in corollary \ref{eig_lambda:cor4}, $\lambda$ becomes unchanged.
Now, we  generalize the fact in the theorem \ref{eig_lambda:th3}. Let $\Sigma_{c_1},\Sigma_{c_2},\ldots ,\Sigma_{c_k}$ be induced subgraphs (not necessarily having disjoint vertex sets) of $\Sigma$. We take distinct vertices $p_1,p_1,\ldots ,p_k\in \Gamma$ and join all vertices $j\in \Sigma_{c_l}$ by an edge to $p_l$, for all $l=1,2,\ldots ,k$ and  produce $\Gamma^\Sigma$. 

\begin{theorem}
\label{eig_lambda:th4}
If there exists a nonzero real valued function f which satisfies
\bel{eig_lambda:equ17}
\frac{1}{n_i}\sum_{j\in \Sigma, j \sim i} f(j)=(1-\lambda)f(i) \text{
  for all }i\in \Sigma \text{ and some real value } \lambda
\qe
and
\bel{eig_lambda:equ18}
\sum_{j\in \Sigma_{c_l}}f(j)=0\text{ for }l=1,2,\ldots ,k,
\qe
where $n_i$ is the degree of $i$ in $\Gamma^\Sigma$,
then $\Gamma^\Sigma$ possesses the eigenvalue $\lambda$ with an eigenfunction $f^{\Gamma^\Sigma}$ which coincides with f on $\Sigma$.
\end{theorem}
\begin{proof}
We construct the eigenfunction
\begin{center}
$f^{\Gamma^\Sigma}(p)=\begin{cases}
f(p) &\text{ if } p\in \Sigma\\
0 &\text{ otherwise.}
\end{cases}$
\end{center}
\end{proof}

Now, we take both $ \Gamma $ and $ \Sigma $ are regular. Let us consider $ V(\Gamma)=\lbrace p_1,p_2,\ldots p_n\rbrace $, $ V(\Sigma)=\displaystyle
{\bigcup_{i=1}^{n}}  V(\Sigma_{c_i}) $  with
$\vert V(\Sigma_{c_i})\vert=\vert V(\Sigma_{c_j})\vert\text{ }\forall i,j$
 and
$\vert\lbrace \Sigma_{c_i}:x\in \Sigma_{c_i}\rbrace\vert=k \text{ }(constant)\text{ }\forall x\in\Sigma$.
Let $\Gamma^\Sigma$ be the graph obtained by joining all vertices of $ \Sigma_{c_i} $ to $p_i$. Thus $ \lbrace \Gamma,\Sigma\rbrace $ becomes an equitable partition of $\Gamma^\Sigma$. Now, the statement of the theorem \ref{eig_lambda:th4} changes for this construction and follows as:
\begin{corollary}
If $ \Sigma  $ be r-regular, then a nonzero real valued function $f$, which satisfies 
$ \frac{1}{r+k}\sum_{j\in \Sigma, j \sim i} f(j)=(1-\lambda)f(i)$,
  for all $i\in \Sigma$ and some real value  $\lambda$ and
$\sum_{j\in \Sigma_{c_l}}f(j)=0$, for $l=1,2,\ldots ,n,$ can be extended to an eigenfunction corresponding to the eigenvalue $ \lambda $ for the graph $ \Gamma^\Sigma. $
\end{corollary}

\section*{Acknowledgements}
The authors gratefully acknowledge the constructive comments of the referee which lead to an improved version of the manuscript. Financial support from Council of Scientific and Industrial Research, India is sincerely acknowledged by Ranjit Mehatari.


\begin{thebibliography}{999}


\bibitem{BanerjeeJost2008a} A. Banerjee, J. Jost. On the spectrum of the normalized graph Laplacian, Linear Algebra Appl. 428(2008) 3015-3022.

\bibitem{BanerjeeJost2009} A. Banerjee, J. Jost. Graph spectra as a systematic tool in computational biology. Discrete Appl. Math 157(2009) 2425-2431.

\bibitem{BanerjeeJosta} A. Banerjee, J. Jost. Laplacian spectrum and protein-protein interaction networks. Preprint. E-print available: arXiv:0705.3373.


\bibitem{BanerjeeJost2007} A. Banerjee, J. Jost. Spectral plots and the representation and interpretation of biological data, Theory in Biosciences 126(2007) 15-21.

\bibitem{BanerjeeJost2008b} A. Banerjee, J. Jost. Spectral plot properties: Towards a qualitative classification of networks, NHM 3(2008) 395-411.

\bibitem{BanerjeeJostc} A. Banerjee, J. Jost, "Spectral characterization of network structures
and dynamics" 
in Dynamics On and Of Complex Networks,pp. 117-132, Modeling and Simulation in Science,
Engineering and Technology,Springer, 2009.



 
\bi{dr} D. M. Cvetkovi´c, M. Doob, H. Sachs, Spectra of Graphs, Theory and Applications,
Third edition, Johann Ambrosius Barth, Heidelberg, 1995.
\bi{Chung} F.Chung, Spectral graph theory, AMS, 1997.

\bibitem{Dutta02}C. Dutta, A. Pan, Horizontal gene transfer and bacterial diversity,  J. Biosci  27(2002) 27-33.

\bibitem{DLS}
G. Gladwell, E. Davies, J. Leydold, and P.Stadler,
 Discrete nodal domain theorems,
Linear Algebra Appl. 336(2001) 51-60.



\bi{HuynenBork1998}M. A. Huynen, P. Bork, Measuring genome evolution, Proc. Nat. Acad. Sc. USA. 95(1998) 5849-5856.

 \bibitem{Jain02} R. Jain, M. C. Rivera, J. E. Moore, J. A. Lake, Horizontal gene transfer in microbial genome evolution, { Theor. Popul. Biol.}  61(2002) 489-495.
\bi{GkantsidisMihailZegura2003}C. Gkantsidis,M. Mihail, E. Zegura, Spectral analysis of Internet topologies, INFOCOM 2003. Twenty-Second Annual Joint Conference of the IEEE Computer and Communications. IEEE Societies , vol.1 364-374, 2003.
\bi{Guo}J.M. Guo, The effect on the Laplacian spectral radius of a grapg by adding or grafting edges, Linear Algebra Appl. 413(2006) 59-71.
\bi{Gutman}I. Gutman, B. Borovi\'canin, "Nullity of graphs: an updated survey" in Selected topics on applications of graph spectra, Math. Inst., Belgrade, 137-154, 2011.

\bi{Merris} R. Merris, Laplacian matrices of graphs -- a survey, Linear Algebra Appl. 198(1994) 143-176.

\bi{Mieghem2011} Piet Van Mieghem, Graph Spectra for Complex Networks, Chambridge University Press, 2011.

\bi{UriAlon}R. Milo, S. Itzkovitz, N. Kashtan, R. Levitt, S. Shen-Orr, I. Ayzenshtat, M. Sheffer, U. Alon, Superfamilies of Evolved and Designed Networks, Science 303(5663)(2004) 1538-1542.

\bi{Pasten}    G. Pasten, O. Rojo, Laplacian spectrum, Laplacian-energy-like invariant, and Kirchhoff index
of graphs constructed by adding edges on pendent vertices, MATCH Commun. Math. Comput. Chem. 73,(2015)  27-40.

\bi{NadakuditiNewman2012} R. R. Nadakuditi, M. E. J. Newman, Graph spectra and the detectability of community structure in networks, Phys. Rev. Lett. 108(2012) 188701.

\bi{ShellmanEtAl} E.R. Shellman, C.F. Burant, S. Schnell, Network motifs provide signatures that characterize metabolism, Mol. Biosyst. 9(2013) 352-360.
\bi{Sole}R.V. Sol\'e, S. Valverde, Are network motifs the spandrels of cellular complexity?, Trends Ecol. Evol. 21(2006) 419-422.
\bi{Vukadinovic2002}D. Vukadinovic, P. Huang, T. Erlebach, On the spectrum and structure of internet topology graphs. Innovative Internet Computing Systems Lecture Notes in Computer Science 2346, 83-95, 2002.


\bi{Wagner2003} A. Wagner,
   How the Global Structure of Protein Interaction Networks Evolves,
   Proc. Royal Soc. B: Biological Sciences 270(2003) 457-466.
\bi{WS} D.Watts, S.Strogatz, Collective dynamics of 'small-world' networks, Nature 393(1998) 440-442.

\end{thebibliography}
\end{document}